\documentclass[11pt]{article}

\usepackage{pgf,acadpgf}
\usepackage{amssymb}
\usepackage{amsmath}
\usepackage{amsthm}
\renewcommand\qedsymbol{$\blacksquare$}
\newcommand\qednoproof{\hfill\qedsymbol}

\numberwithin{equation}{section}

\newtheoremstyle{slplain}
  {\topsep}
  {\topsep}
  {\slshape}
  {0pt}
  {\bfseries}
  {.}
  {0.5em}
  {}

\theoremstyle{slplain}
  \newtheorem{THM}{Theorem}[section]
  \newtheorem{LEM}[THM]{Lemma}
  
  \newtheorem{COR}[THM]{Corollary}

\theoremstyle{definition}
  
  \newtheorem{EX}[THM]{Example}

\newcommand{\alex}{\mathrel{<_{\mathit{alex}}}}
\newcommand\nlongrightarrow{\longrightarrow\kern -1.45em/\kern 0.9em}

\newcommand{\UNION}{\bigcup}

\renewcommand{\le}{\leqslant}
\renewcommand{\ge}{\geqslant}

\newcommand{\0}{\varnothing}

\renewcommand{\phi}{\varphi}
\renewcommand{\epsilon}{\varepsilon}

\newcommand{\NN}{\mathbb{N}}

\newcommand{\QQ}{\mathbb{Q}}

\newcommand{\ZZ}{\mathbb{Z}}

\newcommand{\restr}[2]{\hbox{$#1$}\hbox{$\upharpoonright$}_{#2}}

\newcommand{\Boxed}[1]{\mbox{$#1$}}

\newcommand{\tp}{\mathrm{tp}}

\newcommand{\Emb}{\mathrm{Emb}}

\newcommand{\val}{\mathrm{val}}
\newcommand{\Stp}{\mathrm{Stp}}
\newcommand{\tree}{\mathrm{tree}}
\DeclareRobustCommand{\Stirling}{\genfrac\{\}{0pt}{}}

\title{Countable ordinals and big Ramsey degrees}
\author{%
  Dragan Ma\v sulovi\'c, Branislav \v Sobot\\
  University of Novi Sad, Faculty of Sciences\\
  Department of Mathematics and Informatics\\
  Trg Dositeja Obradovi\'ca 3, 21000 Novi Sad, Serbia\\
  e-mail: $\{$dragan.masulovic,branislav.sobot$\}$@dmi.uns.ac.rs}

\begin{document}
\maketitle

\begin{abstract}
  In this paper we consider big Ramsey degrees of finite chains in countable ordinals.
  We prove that a countable ordinal has finite big Ramsey degrees if and only if it is smaller than $\omega^\omega$.
  Big Ramsey degrees of finite chains in all other countable ordinals are infinite.

  \bigskip


  \noindent \textbf{AMS Subj.\ Classification (2010): 05D10, 03E10} 
\end{abstract}

\section{Introduction}

The infinite version of Ramsey's Theorem claims that
for any $k \ge 2$, $n \ge 1$ and an arbitrary coloring $\chi : \binom \omega n \to k$ of $n$-element subsets of $\omega$
with $k$ colors there exists an infinite $U \subseteq \omega$ such that $\chi(X) = \chi(Y)$ for all $X, Y \in \binom U n$.
In symbols, $\omega \longrightarrow (\omega)^n_k$.
Interestingly, the same is not true for $\QQ$. One can
easily produce a Sierpi\'nski-style coloring of two-element subchains of $\QQ$ with two colors
and with no monochromatic subchain isomorphic to~$\QQ$. So, $\QQ \nlongrightarrow (\QQ)^2_2$.
However, Galvin showed in \cite{galvin1,galvin2} that
for every coloring $\chi : \binom \QQ 2 \to k$ there is an \emph{oligochromatic} copy of $\QQ$
in the following sense: for every coloring of 2-element subsets of $\QQ$ with $k$ colors there is
a $U \subseteq \QQ$ order-isomorphic to $\QQ$ such that the 2-element subsets of $U$ attain at most two colors.
In symbols: $\QQ \longrightarrow (\QQ)^2_{k, 2}$.
This observation was later generalized by Devlin in \cite{devlin}. For each $n \ge 1$ Devlin
found a positive integer $\mathbb{T}_n$ so that
$\QQ \longrightarrow (\QQ)^n_{k,\mathbb{T}_n}$ for every $n \ge 1$ and $k \ge 2$.
The integer $\mathbb{T}_n$ is referred to as the \emph{big Ramsey degree of~$n$ in~$\QQ$}
following Kechris, Pestov and Todor\v cevi\'c \cite{KPT} where big Ramsey degrees were
first introduced in the context of structural Ramsey theory.

In general, an integer $T \ge 1$ is a \emph{big Ramsey degree of a finite chain $n$ in a chain $C$}
if it is the smallest positive integer such that
$$
  C \longrightarrow (C)^n_{k, T} \quad\text{for all $k \ge 2$}.
$$
If no such $T$ exists, we say that \emph{$n$ does not have big Ramsey degree in $C$}.
We denote the big Ramsey degree of $n$ in $C$ by $T(n, C)$, and write
$T(n, C) = \infty$ if $n$ does not have the big Ramsey degree in~$C$.
A chain $C$ \emph{has finite big Ramsey degrees} if $T(n, C) < \infty$ for all integers $n \ge 1$.
In this parlance the infinite version of the Ramsey's theorem takes the following form:

\begin{THM}[Ramsey's Theorem, infinite version]
  $T(n, \omega) = 1$ for every integer $n \ge 1$. \qednoproof
\end{THM}

Recall that a countable chain is \emph{scattered} if it does \emph{not} embed the chain of the rationals~$\QQ$.
Otherwise it is referred to as \emph{non-scattered}.
With Devlin's result \cite{devlin} at hand,
computing big Ramsey degrees of finite chains in countable non-scattered chains is surprisingly easy.
For chains $C$ and $D$ write $C \preccurlyeq D$ if there is an embedding $C \hookrightarrow D$,
and write $C \sim D$ whenever $C \preccurlyeq D$ and $D \preccurlyeq C$. Then it is easy to see that
$C \sim D$ implies that $T(n, C) = T(n, D)$ for all $n \ge 1$.
Therefore, for every non-scattered countable chain $C$ we have that
$T(n, C) = \mathbb{T}_n$, $n \ge 1$, because $C \sim \QQ$.

When it comes to scattered countable chains the situation turns out to
be much more complex. Laver proved in \cite{laver-decomposition} that $T(1, S) < \infty$ for every scattered chain~$S$.
In case of ordinals a proof of a more general result can be found in \cite[p.~189]{Fraisse-relations}:
one first shows that $T(1, \omega^\alpha) = 1$ for every ordinal $\alpha$ and
the Cantor Normal Form Theorem then yields $T(1, \alpha) < \infty$ for every ordinal $\alpha$.

In this paper we consider the big Ramsey degrees $T(n, \alpha)$ for a countable ordinal $\alpha$ and $2 \le n < \omega$.
Our main result is

\bigskip

\noindent
\textbf{Theorem} (Theorem~\ref{fbrd-scat.thm.MAIN})\textbf{.} {\slshape
  Let $\alpha$ be a countable ordinal.
  
  $(a)$ If $\alpha < \omega^\omega$ then $T(n, \alpha) < \infty$ for all $2 \le n < \omega$.

  $(b)$ If $\alpha \ge \omega^\omega$ then $T(n, \alpha) = \infty$ for all $2 \le n < \omega$. \qednoproof}
  
\bigskip

In Section~\ref{fbrd-scat.sec.prelim} we just fix some standard notions and notation.

In Sections~\ref{fbrd-scat.sec.alpha+m}, \ref{fbrd-scat.sec.alpha-cdot-m} and~\ref{fbrd-scat.sec.alpha-pow-m}
we show that if $\alpha$ is a countable ordinal with finite big Ramsey degrees, then so are
$\alpha + m$, $\alpha \cdot m$ and $\alpha^m$ where $m < \omega$.
In Section~\ref{fbrd-scat.sec.ord-pol} we then prove the main result of the paper.

The interest in Ramsey degrees $T(n, \omega^m)$ is far from new. It was established quite some time ago
that:

\begin{THM}\label{fbrd-scat.thm.williams}
  $T(n, \omega^m) < \infty$ for all $1 \le n, m < \omega$. \qednoproof
\end{THM}

\noindent
A proof of this in case $n = 2$ (which easily generalizes to other values of $n$)
can be found in~\cite[Theorem 7.2.7]{williams}. Moreover, Galvin established that the sequence
$T(n, \omega^2)$ coincides with the OEIS sequence A000311, a sequence first studied by Ernst Schr\"oder in 1870.
He was motivated by the observation that, if $\mathcal U$ is a free ultrafilter on $\omega$, and if $t = T(n, \omega^2)$,
then the partition relation $\omega \longrightarrow (\mathcal U)^n_{t, t-1}$ implies that $\mathcal U$ is a $P$-point~\cite{Galvin-private-comm}.
Galvin's strategy can be applied to compute the values of $T(n, \omega^m)$ for all finite~$n$ and~$m$.

If one wants to show that $T(n, \alpha) < \infty$ for an ordinal~$\alpha < \omega^\omega$
an obvious line of attack would be to start from Theorem~\ref{fbrd-scat.thm.williams}
and, working ``bottom-up,'' propagate the property of having finite big Ramsey degrees to finite sums of finite powers of~$\omega$
(since countable ordinals smaller than $\omega^\omega$ can be expressed
as $\omega^{d_0} \cdot c_0 + \omega^{d_1} \cdot c_1 + \ldots + \omega^{d_{k-1}} \cdot c_{k-1}$
where both $d_0 > d_1 > \ldots > d_{k-1} \ge 0$ and $c_0, c_1, \ldots, c_{k-1} \ge 1$ are integers).
Unfortunately, we were unable to do that. We, however, managed to apply the ``top-down'' approach because
it is easy to show that the property of having finite big Ramsey degrees propagates
from finite sums of finite powers of~$\omega$ to their subsums (Lemma~\ref{fbrd-scat.lem.sub-sum}).
Having shown that $\alpha+m$, $\alpha \cdot m$ and $\alpha^m$, $m < \omega$, have finite big Ramsey degrees
whenever this holds for a countable ordinal $\alpha$ (Theorems~\ref{fbrd-scat.thm.alpha+m}, \ref{fbrd-scat.thm.alpha-cdot-m}
and~\ref{fbrd-scat.thm.alpha-pow-m}), we conclude that every ordinal of the form $(\omega \cdot c + 1)^d$, where $c, d < \omega$,
has finite big Ramsey degrees. For appropriately chosen $c$ and $d$ we expand this expression and pass to a subsum
to get the result for an arbitrary ordinal~$\alpha < \omega^\omega$.

We then move on to show that $T(n, \alpha) = \infty$ whenever $n \ge 2$ and $\alpha \ge \omega^\omega$.
Our starting point is an unpublished result of Galvin about square bracket partition relations
which express ``strong counterexamples'' to ordinary partition relations.
For chains $C$, $D_0$, $D_1$, $D_2$, \ldots, and $n < \omega$ write
$$
  C \longrightarrow [D_0, D_1, D_2, \ldots]^n
$$
to denote that for every coloring $\chi : \Emb(n, C) \to \omega$ there is an $i < \omega$ and
a subchain $S \subseteq C$ such that $S \cong D_i$ and $i \notin \chi(\Emb(n, S))$.

Erd\H os and Hajnal note in~\cite[p.~275]{erdos-hajnal} that in 1971
Galvin proved the following:

\begin{THM}[Galvin 1971]\label{fbrd-scat.thm.bracket}
  If $S$ is a scattered chain that contains no uncountable well-ordered subsets then
  $S \not\longrightarrow [\omega, \omega^2, \omega^2, \omega^3, \omega^3, \ldots]^2.$ \qednoproof
\end{THM}

\noindent
A proof of this result can be found in~\cite[p.~234]{Todorcevic-OscInts}.
As an immediate consequence of Theorem~\ref{fbrd-scat.thm.bracket} we get that
$T(n, \omega^\alpha) = \infty$ for every countable ordinal~$\alpha \ge \omega$
(Lemma~\ref{fbrd-scat.lem.infty}). The conclusion then follows by another application of Lemma~\ref{fbrd-scat.lem.sub-sum}.

In some fairly simple cases we were able to compute the exact big Ramsey degrees.
In Section~\ref{fbrd-scat.sec.alpha+m} we show that
$T(n, \omega + m) = \sum_{j=0}^n \binom mj$ for all $n \ge 2$ and $m \ge 1$, while
in Section~\ref{fbrd-scat.sec.alpha-cdot-m} we show that $T(n, \omega \cdot m) = m^n$.
As a spin-off in Section~\ref{fbrd-scat.sec.alpha-cdot-m}
we prove Corollary~\ref{fbrd-scat.cor.inf-prod-ramsey} which we see as an
infinite analogue of the Finite Product Ramsey Theorem:

\begin{THM}[Finite Product Ramsey Theorem \cite{GRS}]
  For every choice of integers $p \ge 1$, $n_0, \ldots, n_{p-1} \ge 1$, $m_0 \ge n_0$, \ldots, $m_{p-1} \ge n_{p-1}$, and $k \ge 2$,
  there are integers $\ell_0 \ge m_0$, \ldots, $\ell_{p-1} \ge m_{p-1}$
  such that for all finite sets $L_0$ of size $\ell_0$, \ldots, $L_{p-1}$ of size $\ell_{p-1}$
  and for every coloring $\chi : \binom{L_0}{n_0} \times \ldots \times \binom{L_{p-1}}{n_{p-1}} \to k$
  there exist $M_0 \subseteq L_0$ of size $m_0$, \ldots, $M_{p-1} \subseteq L_{p-1}$ of size $m_{p-1}$
  such that $\chi(A_0, \ldots, A_{p-1}) = \chi(B_0, \ldots, B_{p-1})$
  for all $(A_0, \ldots, A_{p-1}), (B_0, \ldots, B_{p-1}) \in \binom{M_0}{n_0} \times \ldots \times \binom{M_{p-1}}{n_{p-1}}$.
  \qednoproof
\end{THM}

\noindent
It is a well-known fact (see \cite{GRS}) that the analogue of the above theorem fails in the infinite case.
One can easily construct a coloring $\chi : \binom \omega 1 \times \binom \omega 1 \to 2$ such that
no infinite subsets $S, T \subseteq \omega$ have the property that $\binom S 1 \times \binom T 1$ is monochromatic.
However, we can prove that for any choice of integers $p \ge 1$, $n_0, \ldots, n_{p-1} \ge 1$ and for any coloring
$\chi : \binom{\omega}{n_0} \times \ldots \times \binom{\omega}{n_{p-1}} \to k$
there is an infinite $S \subseteq \omega$ such that on the set
$\binom{S}{n_0} \times \ldots \times \binom{S}{n_{p-1}}$ the coloring $\chi$ attains at most
$\sum_{j=1}^N j! \Stirling Nj$ colors, where $N = n_0 + n_1 + \ldots + n_{s-1}$ and $\Stirling nk$ is the Stirling
number of the second kind. To put it briefly, although in the infinite case we do not necessarily have a monochromatic
subset, we can always prove the existence of an oligochromatic one.

We have seen that big Ramsey degrees of finite chains can be computed in $\omega$ (infinite version of Ramsey's theorem)
and in $\QQ$ (Devlin's result~\cite{devlin}). In Section~\ref{fbrd-scat.sec.conclusion} we do justice to another famous countable chain:
we compute big Ramsey degrees of finite chains in~$\ZZ$.

\section{Preliminaries}
\label{fbrd-scat.sec.prelim}

A \emph{chain} is a pair $(A, \Boxed<)$ where $<$ is a linear order on~$A$.
If we wish to stress that two chains are isomorphic as ordered sets we shall say that they are
\emph{order-isomorphic}.

Let $(A_i, \Boxed{<_i})$ be chains, $i < k$.
The linear orders $<_i$, $i < k$, induce the \emph{antilexicographic} order $\alex$
on $A_0 \times \ldots \times A_{k-1}$ as follows:
$(a_0, \ldots, a_{k-1}) \alex (b_0, \ldots, b_{k-1})$ iff
there is an $s < k$ such that $a_s \mathrel{<_s} b_s$, and $a_j = b_j$ for all $j > s$.

For a well-ordered set $A$ let $\tp(A)$ denote the \emph{order type} of $A$, that is, the unique ordinal $\alpha$
which is order-isomorphic to $A$.
Let $(A_\xi)_{\xi \in I}$ be a sequence of well-ordered sets indexed by a well-ordered set~$I$.
Then by $\sum_{\xi \in I} A_\xi$ we denote the well-ordered set
$(\UNION_{\xi \in I} A_\xi \times \{\xi\}, \Boxed{\alex})$.
If $I = m$ we shall simply write $A_0 + \ldots + A_{m-1}$ instead of $\sum_{\xi \in m} A_\xi$.
In particular, $A \cdot m = \underbrace{A + \ldots + A}_m$.

For another sequence $(B_\xi)_{\xi \in I}$ of well-ordered sets indexed by~$I$ and for a sequence
$(f_\xi)_{\xi \in I}$ of maps $f_\xi : A_\xi \to B_\xi$, $\xi \in I$, there is a unique
map
$
  \sum_{\xi \in I} f_\xi : \sum_{\xi \in I} A_\xi \to \sum_{\xi \in I} B_\xi
$
which takes $(\gamma, \xi)$ to $(f_\xi(\gamma), \xi)$.

For well-ordered sets $A$ and $B$ we let $A \cdot B = \tp(A \times B, \Boxed{\alex})$
and $A^m = (\underbrace{A \times \ldots \times A}_m, \Boxed{\alex})$.

If $(\alpha_\xi)_{\xi \in I}$ is a sequence of ordinals indexed by a well-ordered set~$I$
then $\sum_{\xi \in I} \alpha_\xi = \tp\left(\UNION_{\xi \in I} \alpha_\xi \times \{\xi\}, \Boxed{\alex}\right)$.
If $I = m$ we shall simply write $\alpha_0 + \ldots + \alpha_{m-1}$ instead of $\sum_{\xi \in m} \alpha_\xi$.
For ordinals $\alpha$ and $\beta$ we have that $\alpha \cdot \beta = \tp(\alpha \times \beta, \Boxed{\alex})$
and $\alpha^m = \underbrace{\alpha \cdot \ldots \cdot \alpha}_m$.

For the sake of simplicity we use the same notation for the operations on well-ordered sets and for the
operations on ordinals. Moreover, in some
proofs we shall move freely between $\sum_{\xi \in I} \alpha_\xi$ and $\left(\UNION_{\xi \in I} \alpha_\xi \times \{\xi\}, \Boxed{\alex}\right)$,
and between $\alpha \cdot \beta$ and $(\alpha \times \beta, \Boxed{\alex})$.
We believe that the context will always be sufficient to enable the correct parsing of the symbols.

A \emph{total quasiorder} is a reflexive and transitive binary relation such that each pair of elements
of the underlying set is comparable. Each total quasiorder $\sigma$ on a set $I$ induces an equivalence relation $\equiv_\sigma$
on $I$ and a linear order $\prec_\sigma$ on $I / \Boxed{\equiv_\sigma}$ in a natural way: $i \mathrel{\equiv_\sigma} j$ if
$(i, j) \in \sigma$ and $(j, i) \in \sigma$, while $(i / \Boxed{\equiv_\sigma}) \mathrel{\prec_\sigma} (j / \Boxed{\equiv_\sigma})$
if $(i, j) \in \sigma$ and $(j, i) \notin \sigma$.

Let $C$ be a chain and $n$ a finite chain. Then the set of all the $n$-element subchains of $C$
clearly corresponds to the set $\Emb(n, C)$ of all the embeddings $n \hookrightarrow C$. This is why we
formally introduce big Ramsey degrees as follows. For chains $A$, $B$, $C$ and integers $k \ge 2$ and $t \ge 1$
we write $C \longrightarrow (B)^{A}_{k, t}$ to denote that for every $k$-coloring $\chi : \Emb(A, C) \to k$
there is an embedding $w \in \Emb(B, C)$ such that $|\chi(w \circ \Emb(A, B))| \le t$.
For a chain $C$ and a finite chain $n$ we say that $n$ has \emph{finite big Ramsey degree in $C$}
if there exists a positive integer $t$ such that for each $k \ge 2$ we have that
$C \longrightarrow (C)^{n}_{k, t}$.
Equivalently, a finite chain $n$ has finite big Ramsey degree in a chain $C$
if there exists a positive integer $t$ such that for every $k \ge 2$ and every
$k$-coloring $\chi : \Emb(n, C) \to k$ there is a $U \subseteq C$ order-isomorphic to $C$
such that $|\chi(\Emb(n, U))| \le t$.
The least such $t$ is then denoted by $T(n, C)$. If such a $t$ does not exist
we say that $A$ \emph{does not have finite big Ramsey degree in $C$} and write
$T(A, C) = \infty$.
We say that a chain $C$ \emph{has finite big Ramsey degrees}
if $T(n, C) < \infty$ for all $n \ge 1$.
For any chain $C$ we let $T(0, C) = 1$ by definition.

Let us show that $T$ is monotonous in the first argument whenever the big Ramsey degrees are
calculated in a limit ordinal.

\begin{LEM}\label{fbrd-scat.lem.T-monotono}
  Let $\alpha$ be a limit ordinal and $m, n < \omega$. If $m \le n$ then $T(m, \alpha) \le T(n, \alpha)$.
\end{LEM}
\begin{proof}
  Let $T(n, \alpha) = t \in \NN$.
  Take any $k \ge 2$ and let $\chi : \Emb(m, \alpha) \to k$ be a coloring.
  Define $\chi' : \Emb(n, \alpha) \to k$ by $\chi'(h) = \chi(\restr hm)$.
  Then there is an $S \subseteq \alpha$ order-isomorphic to $\alpha$ such that
  $|\chi'(\Emb(n, S))| \le t$. Since $S$ is order-isomorphic to a limit ordinal, every
  $m$-element subchain of $S$ can be extended to an $n$-element subchain of $S$, whence
  $\chi(\Emb(m, S)) \subseteq \chi'(\Emb(n, S))$. Therefore,
  $|\chi(\Emb(n, S))| \le t$.
\end{proof}

\section{Adding a finite ordinal}
\label{fbrd-scat.sec.alpha+m}

For a well-ordered set $A$ and an embedding $f : n \hookrightarrow A + m$
recall that $f(i) = (a, 0)$ for some $a \in A$ or $f(i) = (j, 1)$ for some $j \in m$. Let
$$
  \tp(f) = \{j \in m : (\exists i \in n) f(i) = (j, 1)\}.
$$
denote the \emph{additive type of $f$}. For an additive type $\tau \subseteq m$ let
$$
  \Emb_\tau(n, A + m) = \{ f \in \Emb(n, A + m) : \tp(f) = \tau \}.
$$

\begin{LEM}\label{fbrd-scat.lem.1-additive}
  Let $A$ be a countable well-ordered set with finite big Ramsey degrees.
  Fix integers $n \ge 1$ and $m \ge 1$.
  For every additive type $\tau \subseteq m$ with $|\tau| \le n$, every $k \ge 2$ and every coloring
  $
    \chi : \Emb_\tau(n, A + m) \to k
  $
  there is a $U \subseteq A$ order-isomorphic to $A$ such that
  $
    |\chi(\Emb_\tau(n, U + m))| \le T(n - |\tau|, A)
  $.
  (Recall that $T(0, A) = 1$ by definition.)
\end{LEM}
\begin{proof}
  Assume, first, that $|\tau| = n$. Then for every $U \subseteq A$ we have that
  $|\Emb_\tau(n, U + m)| = 1$, whence $|\chi(\Emb_\tau(n, U + m))| = 1 = T(0, A)$.
  So, let $s = |\tau| < n$ and let
  $
    \Phi : \Emb_\tau(n, A + m) \to \Emb(n - s, A)
  $
  be the bijection that takes $f \in \Emb_\tau(n, A + m)$ to $\restr{f}{n-s} \in \Emb(n - s, A)$.
  Fix a $k \ge 2$ and a coloring
  $
    \chi : \Emb_\tau(n, A + m) \to k.
  $
  Let
  $
    \chi' : \Emb(n - s, A) \to k
  $
  be the coloring defined by $\chi'(f) = \chi(\Phi^{-1}(f))$. Then by the assumption that $A$ has finite big Ramsey degrees
  there is a $U \subseteq A$ order-isomorphic to $A$ such that
  $
    |\chi'(\Emb(n - s, U))| \le T(n - s, A)
  $.
  But then it easily follows that $|\chi(\Emb_\tau(n, U + m))| \le T(n - s, A)$.
\end{proof}

\begin{THM}\label{fbrd-scat.thm.alpha+m}
  Let $\alpha$ be a countable ordinal with finite big Ramsey degrees.
  Then $\alpha + m$ has finite big Ramsey degrees for all $m \ge 1$. More precisely,
  $$
    T(n, \alpha + m) \le \sum_{j=0}^n \binom mj \cdot T(n - j, \alpha),
  $$
  where we take $T(0, \alpha) = 1$ and $\binom mj = 0$ whenever $m < j$.
\end{THM}
\begin{proof}
  Let $Q = \{ \tau \subseteq m : |\tau| \le n \}$ be the set of all the additive types realized by
  members of $\Emb(n, \alpha + m)$. Let $Q = \{\tau_0, \tau_1, \ldots, \tau_{t-1}\}$ so that $|Q| = t$.
  Fix a $k \ge 2$ and a coloring
  $
    \chi : \Emb(n, \alpha + m) \to k
  $.
  By Lemma~\ref{fbrd-scat.lem.1-additive} there is a $U_0 \subseteq \alpha$ order-isomorphic to $\alpha$ such that
  $$
    |\chi(\Emb_{\tau_0}(n, U_0 + m))| \le T(n - |\tau_0|, \alpha).
  $$
  By the same lemma for each $j \in \{1, \ldots, t-1\}$ we then inductively obtain
  a $U_j \subseteq U_{j-1}$ order-isomorphic to $U_{j-1}$ (and hence to $\alpha$) such that
  $$
    |\chi(\Emb_{\tau_j}(n, U_j + m))| \le T(n - |\tau_j|, \alpha).
  $$
  Then, using the fact that $U_{t-1} \subseteq U_j$ we have that
  \begin{align*}
    |\chi(\Emb(n, U_{t-1} + m))|
    &= \sum_{j < t} |\chi(\Emb_{\tau_j}(n, U_{t-1} + m))|\\
    &\le \sum_{j < t} |\chi(\Emb_{\tau_j}(n, U_{j} + m))|\\
    &\le \sum_{j < t} T(n - |\tau_{j}|, \alpha) = \sum_{j=0}^n \binom mj \cdot T(n - j, \alpha). \qedhere
  \end{align*}
\end{proof}

\begin{COR}
  For all $m \ge 1$ and $n \ge 1$ we have that
  $$
    T(n, \omega + m) = \sum_{j=0}^n \binom mj,
  $$
  where we take $\binom mj = 0$ whenever $m < j$. In particular, if $n \ge m$ then
  $T(n, \omega + m) = 2^m$.
\end{COR}
\begin{proof}
  It immediately follows from Theorem~\ref{fbrd-scat.thm.alpha+m} that
  $$
    T(n, \omega + m) \le \sum_{j=0}^n \binom mj
  $$
  since $T(j, \omega) = 1$ for all $j \ge 0$.
  In order to conclude the proof we have to show that there exists a $k \ge t$ and a coloring
  $
    \chi : \Emb(n, \omega + m) \to k
  $
  such that $|\chi(\Emb(n, U + m))| \ge t$ for every infinite $U \subseteq \omega$, but that is straightforward.
  Take $k = t$ and consider
  $$
    \chi^* : \Emb(n, \omega + m) \to t
  $$
  such that $\chi^*(f) = j$ if and only if $\tp(f) = \tau_j$. Then it is easy to see that
  for every infinite $U \subseteq \omega$ members of $\Emb(n, U + m)$ realize all the types from $Q$,
  so $|\chi^*(\Emb(n, U + m))| = t$.
\end{proof}

\section{Multiplying by a finite ordinal}
\label{fbrd-scat.sec.alpha-cdot-m}

Let $A$ be a well-ordered set and let $f : n \hookrightarrow A \cdot m$ be an embedding.
For each $i < n$ we take that $f(i) = (a, \ell)$ where $a \in \alpha$ and $\ell \in m$. Therefore, we refer to
$\pi_0(f(i)) = a$ as the \emph{value} of $f(i)$, and to $\pi_1(f(i)) = \ell$ as the \emph{level} of $f(i)$.
Let
$$
  \val(f) = \restr{A}{\{\pi_0(f(i)) : i \in n\}}
$$
denote the subchain of $A$ induced by the values of $f$. The \emph{multiplicative type of $f$} is a tuple
$$
  \tp(f) = (p_0, p_1, \dots, p_{m-1}, \sigma),
$$
where
\begin{itemize}
\item
  $p_\ell$ is the number of those $i < n$ such that $f$ takes $i$ to the $\ell$th level
  (note that $0 \le p_\ell \le n$ and $p_0 + p_1 + \dots + p_{m-1} = n$); and
\item
  $\sigma$ is a total quasiorder on $n$ defined by
  $
    (i, j) \in \sigma \text{ iff } \pi_0(f(i)) \le \pi_0(f(j))
  $.
\end{itemize}

A tuple $\tau = (p_0, p_1, \dots, p_{m-1}, \sigma)$ is an \emph{$(n, m)$-multiplicative type}
if $\tau = \tp(f)$ for some embedding $f : n \hookrightarrow A \cdot m$.
Clearly, given finite $n$ and $m$, there are only finitely many $(n, m)$-multiplicative types.
The number $|n / \Boxed{\equiv_\sigma}|$ will be referred to as the \emph{rank of $\tau$}
and denoted by $r(\tau)$.
Note that if $\tp(f) = \tau$ then $r(\tau)$ is equal to the length of the chain~$\val(f)$.
Given finite $n$ and $m$, each embedding $f : n \hookrightarrow A \cdot m$ is uniquely determined
by the pair $(\tp(f), \val(f))$. For an $(n, m)$-multiplicative type $\tau$ let
$$
  \Emb_\tau(n, A \cdot m) = \{ f \in \Emb(n, A \cdot m) : \tp(f) = \tau \}.
$$

\begin{figure}
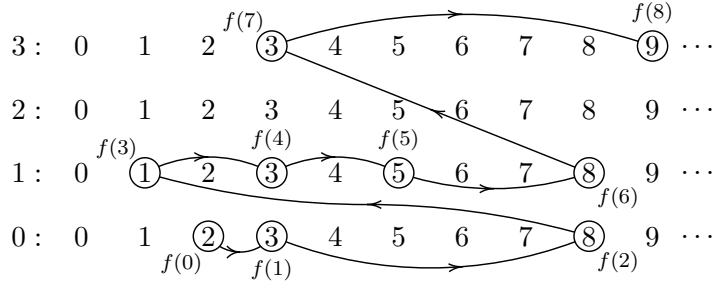

  \newcommand\scrs\scriptsize
  \centering
\begin{pgfpicture}
  \pgfsetxvec{\pgfpoint{\acadpgfunit}{0pt}}
  \pgfsetyvec{\pgfpoint{0pt}{\acadpgfunit}}
  \pgfsetlinewidth{\acadpgflinewidth}
  \pgftransformshift{\pgfpointxy{-3.31158}{-99.0811}}

  \begin{pgfscope}
    \pgfpathellipse{\pgfpointxy{350.0}{250.0}}{\pgfpointxy{23.8122}{0.0}}{\pgfpointxy{0.0}{23.8122}}
    \pgfusepath{stroke}
  \end{pgfscope}
  \begin{pgfscope}
    \pgfpathellipse{\pgfpointxy{450.0}{250.0}}{\pgfpointxy{23.8122}{0.0}}{\pgfpointxy{0.0}{23.8122}}
    \pgfusepath{stroke}
  \end{pgfscope}
  \begin{pgfscope}
    \pgfpathellipse{\pgfpointxy{950.0}{250.0}}{\pgfpointxy{23.8122}{0.0}}{\pgfpointxy{0.0}{23.8122}}
    \pgfusepath{stroke}
  \end{pgfscope}
  \begin{pgfscope}
    \pgfpathellipse{\pgfpointxy{250.0}{350.0}}{\pgfpointxy{23.8122}{0.0}}{\pgfpointxy{0.0}{23.8122}}
    \pgfusepath{stroke}
  \end{pgfscope}
  \begin{pgfscope}
    \pgfpathellipse{\pgfpointxy{450.0}{350.0}}{\pgfpointxy{23.8122}{0.0}}{\pgfpointxy{0.0}{23.8122}}
    \pgfusepath{stroke}
  \end{pgfscope}
  \begin{pgfscope}
    \pgfpathellipse{\pgfpointxy{650.0}{350.0}}{\pgfpointxy{23.8122}{0.0}}{\pgfpointxy{0.0}{23.8122}}
    \pgfusepath{stroke}
  \end{pgfscope}
  \begin{pgfscope}
    \pgfpathellipse{\pgfpointxy{950.0}{350.0}}{\pgfpointxy{23.8122}{0.0}}{\pgfpointxy{0.0}{23.8122}}
    \pgfusepath{stroke}
  \end{pgfscope}
  \begin{pgfscope}
    \pgfpathellipse{\pgfpointxy{450.0}{550.0}}{\pgfpointxy{23.8122}{0.0}}{\pgfpointxy{0.0}{23.8122}}
    \pgfusepath{stroke}
  \end{pgfscope}
  \begin{pgfscope}
    \pgfpathellipse{\pgfpointxy{1050.0}{550.0}}{\pgfpointxy{23.8122}{0.0}}{\pgfpointxy{0.0}{23.8122}}
    \pgfusepath{stroke}
  \end{pgfscope}
  \begin{pgfscope}
    \pgfpathmoveto{\pgfpointxy{927.078}{256.449}}
    \pgfpatharcaxes{74.8313}{90.0}{\pgfpointxy{1250.0}{0.0}}{\pgfpointxy{0.0}{1250.0}}
    \pgfusepath{stroke}
  \end{pgfscope}
  \begin{pgfscope}
    \pgfpathmoveto{\pgfpointxy{272.922}{343.551}}
    \pgfpatharcaxes{254.831}{270.0}{\pgfpointxy{1250.0}{0.0}}{\pgfpointxy{0.0}{1250.0}}
    \pgfusepath{stroke}
  \end{pgfscope}
  \begin{pgfscope}
    \pgfpathmoveto{\pgfpointxy{927.891}{358.844}}
    \pgfpathlineto{\pgfpointxy{472.109}{541.156}}
    \pgfusepath{stroke}
  \end{pgfscope}
  \begin{pgfscope}
    \pgfpathmoveto{\pgfpointxy{367.655}{234.021}}
    \pgfpatharcaxes{238.834}{301.166}{\pgfpointxy{62.5}{0.0}}{\pgfpointxy{0.0}{62.5}}
    \pgfusepath{stroke}
  \end{pgfscope}
  \begin{pgfscope}
    \pgfpathmoveto{\pgfpointxy{472.145}{241.246}}
    \pgfpatharcaxes{249.479}{290.521}{\pgfpointxy{650.0}{0.0}}{\pgfpointxy{0.0}{650.0}}
    \pgfusepath{stroke}
  \end{pgfscope}
  \begin{pgfscope}
    \pgfpathmoveto{\pgfpointxy{428.394}{360.011}}
    \pgfpatharcaxes{68.3513}{111.649}{\pgfpointxy{212.5}{0.0}}{\pgfpointxy{0.0}{212.5}}
    \pgfusepath{stroke}
  \end{pgfscope}
  \begin{pgfscope}
    \pgfpathmoveto{\pgfpointxy{628.394}{360.011}}
    \pgfpatharcaxes{68.3513}{111.649}{\pgfpointxy{212.5}{0.0}}{\pgfpointxy{0.0}{212.5}}
    \pgfusepath{stroke}
  \end{pgfscope}
  \begin{pgfscope}
    \pgfpathmoveto{\pgfpointxy{672.716}{342.86}}
    \pgfpatharcaxes{254.026}{285.974}{\pgfpointxy{462.5}{0.0}}{\pgfpointxy{0.0}{462.5}}
    \pgfusepath{stroke}
  \end{pgfscope}
  \begin{pgfscope}
    \pgfpathmoveto{\pgfpointxy{1027.38}{557.432}}
    \pgfpatharcaxes{72.5504}{107.45}{\pgfpointxy{925.0}{0.0}}{\pgfpointxy{0.0}{925.0}}
    \pgfusepath{stroke}
  \end{pgfscope}
  \begin{pgfscope}
    \pgfpathmoveto{\pgfpointxy{728.441}{206.531}}
    \pgfpatharcaxes{243.143}{274.524}{\pgfpointxy{43.0198}{0.0}}{\pgfpointxy{0.0}{43.0198}}
    \pgfusepath{stroke}
  \end{pgfscope}
  \begin{pgfscope}
    \pgfpathmoveto{\pgfpointxy{751.27}{202.025}}
    \pgfpatharcaxes{94.524}{125.905}{\pgfpointxy{43.0198}{0.0}}{\pgfpointxy{0.0}{43.0198}}
    \pgfusepath{stroke}
  \end{pgfscope}
  \begin{pgfscope}
    \pgfpathmoveto{\pgfpointxy{379.292}{234.607}}
    \pgfpatharcaxes{220.227}{270.0}{\pgfpointxy{27.1231}{0.0}}{\pgfpointxy{0.0}{27.1231}}
    \pgfusepath{stroke}
  \end{pgfscope}
  \begin{pgfscope}
    \pgfpathmoveto{\pgfpointxy{400.0}{225.0}}
    \pgfpatharcaxes{90.0}{139.773}{\pgfpointxy{27.1231}{0.0}}{\pgfpointxy{0.0}{27.1231}}
    \pgfusepath{stroke}
  \end{pgfscope}
  \begin{pgfscope}
    \pgfpathmoveto{\pgfpointxy{622.45}{293.837}}
    \pgfpatharcaxes{59.2966}{90.0}{\pgfpointxy{43.9691}{0.0}}{\pgfpointxy{0.0}{43.9691}}
    \pgfusepath{stroke}
  \end{pgfscope}
  \begin{pgfscope}
    \pgfpathmoveto{\pgfpointxy{600.0}{300.0}}
    \pgfpatharcaxes{270.0}{300.703}{\pgfpointxy{43.9691}{0.0}}{\pgfpointxy{0.0}{43.9691}}
    \pgfusepath{stroke}
  \end{pgfscope}
  \begin{pgfscope}
    \pgfpathmoveto{\pgfpointxy{327.844}{381.9}}
    \pgfpatharcaxes{235.404}{270.0}{\pgfpointxy{39.0217}{0.0}}{\pgfpointxy{0.0}{39.0217}}
    \pgfusepath{stroke}
  \end{pgfscope}
  \begin{pgfscope}
    \pgfpathmoveto{\pgfpointxy{350.0}{375.0}}
    \pgfpatharcaxes{90.0}{124.596}{\pgfpointxy{39.0217}{0.0}}{\pgfpointxy{0.0}{39.0217}}
    \pgfusepath{stroke}
  \end{pgfscope}
  \begin{pgfscope}
    \pgfpathmoveto{\pgfpointxy{527.844}{381.9}}
    \pgfpatharcaxes{235.404}{270.0}{\pgfpointxy{39.0217}{0.0}}{\pgfpointxy{0.0}{39.0217}}
    \pgfusepath{stroke}
  \end{pgfscope}
  \begin{pgfscope}
    \pgfpathmoveto{\pgfpointxy{550.0}{375.0}}
    \pgfpatharcaxes{90.0}{124.596}{\pgfpointxy{39.0217}{0.0}}{\pgfpointxy{0.0}{39.0217}}
    \pgfusepath{stroke}
  \end{pgfscope}
  \begin{pgfscope}
    \pgfpathmoveto{\pgfpointxy{777.642}{331.406}}
    \pgfpatharcaxes{238.026}{270.0}{\pgfpointxy{42.2214}{0.0}}{\pgfpointxy{0.0}{42.2214}}
    \pgfusepath{stroke}
  \end{pgfscope}
  \begin{pgfscope}
    \pgfpathmoveto{\pgfpointxy{800.0}{325.0}}
    \pgfpatharcaxes{90.0}{121.974}{\pgfpointxy{42.2214}{0.0}}{\pgfpointxy{0.0}{42.2214}}
    \pgfusepath{stroke}
  \end{pgfscope}
  \begin{pgfscope}
    \pgfpathmoveto{\pgfpointxy{718.652}{436.046}}
    \pgfpatharcaxes{38.1986}{68.1986}{\pgfpointxy{45.0}{0.0}}{\pgfpointxy{0.0}{45.0}}
    \pgfusepath{stroke}
  \end{pgfscope}
  \begin{pgfscope}
    \pgfpathmoveto{\pgfpointxy{700.0}{450.0}}
    \pgfpatharcaxes{248.199}{278.199}{\pgfpointxy{45.0}{0.0}}{\pgfpointxy{0.0}{45.0}}
    \pgfusepath{stroke}
  \end{pgfscope}
  \begin{pgfscope}
    \pgfpathmoveto{\pgfpointxy{727.568}{606.212}}
    \pgfpatharcaxes{239.042}{270.0}{\pgfpointxy{43.6074}{0.0}}{\pgfpointxy{0.0}{43.6074}}
    \pgfusepath{stroke}
  \end{pgfscope}
  \begin{pgfscope}
    \pgfpathmoveto{\pgfpointxy{750.0}{600.0}}
    \pgfpatharcaxes{90.0}{120.958}{\pgfpointxy{43.6074}{0.0}}{\pgfpointxy{0.0}{43.6074}}
    \pgfusepath{stroke}
  \end{pgfscope}
  \pgftext[at={\pgfpointxy{950.0}{250.0}}]{8}
  \pgftext[at={\pgfpointxy{850.0}{250.0}}]{7}
  \pgftext[at={\pgfpointxy{750.0}{250.0}}]{6}
  \pgftext[at={\pgfpointxy{650.0}{250.0}}]{5}
  \pgftext[at={\pgfpointxy{550.0}{250.0}}]{4}
  \pgftext[at={\pgfpointxy{450.0}{250.0}}]{3}
  \pgftext[at={\pgfpointxy{350.0}{250.0}}]{2}
  \pgftext[at={\pgfpointxy{250.0}{250.0}}]{1}
  \pgftext[at={\pgfpointxy{150.0}{250.0}}]{0}
  \pgftext[at={\pgfpointxy{1125.0}{250.0}}]{$\dots$}
  \pgftext[right,at={\pgfpointxy{88.0}{250.0}}]{$0:$}
  \pgftext[at={\pgfpointxy{950.0}{350.0}}]{8}
  \pgftext[at={\pgfpointxy{850.0}{350.0}}]{7}
  \pgftext[at={\pgfpointxy{750.0}{350.0}}]{6}
  \pgftext[at={\pgfpointxy{650.0}{350.0}}]{5}
  \pgftext[at={\pgfpointxy{550.0}{350.0}}]{4}
  \pgftext[at={\pgfpointxy{450.0}{350.0}}]{3}
  \pgftext[at={\pgfpointxy{350.0}{350.0}}]{2}
  \pgftext[at={\pgfpointxy{250.0}{350.0}}]{1}
  \pgftext[at={\pgfpointxy{150.0}{350.0}}]{0}
  \pgftext[at={\pgfpointxy{1125.0}{350.0}}]{$\dots$}
  \pgftext[right,at={\pgfpointxy{88.0}{350.0}}]{$1:$}
  \pgftext[at={\pgfpointxy{950.0}{450.0}}]{8}
  \pgftext[at={\pgfpointxy{850.0}{450.0}}]{7}
  \pgftext[at={\pgfpointxy{750.0}{450.0}}]{6}
  \pgftext[at={\pgfpointxy{650.0}{450.0}}]{5}
  \pgftext[at={\pgfpointxy{550.0}{450.0}}]{4}
  \pgftext[at={\pgfpointxy{450.0}{450.0}}]{3}
  \pgftext[at={\pgfpointxy{350.0}{450.0}}]{2}
  \pgftext[at={\pgfpointxy{250.0}{450.0}}]{1}
  \pgftext[at={\pgfpointxy{150.0}{450.0}}]{0}
  \pgftext[at={\pgfpointxy{1125.0}{450.0}}]{$\dots$}
  \pgftext[right,at={\pgfpointxy{88.0}{450.0}}]{$2:$}
  \pgftext[at={\pgfpointxy{950.0}{550.0}}]{8}
  \pgftext[at={\pgfpointxy{850.0}{550.0}}]{7}
  \pgftext[at={\pgfpointxy{750.0}{550.0}}]{6}
  \pgftext[at={\pgfpointxy{650.0}{550.0}}]{5}
  \pgftext[at={\pgfpointxy{550.0}{550.0}}]{4}
  \pgftext[at={\pgfpointxy{450.0}{550.0}}]{3}
  \pgftext[at={\pgfpointxy{350.0}{550.0}}]{2}
  \pgftext[at={\pgfpointxy{250.0}{550.0}}]{1}
  \pgftext[at={\pgfpointxy{150.0}{550.0}}]{0}
  \pgftext[at={\pgfpointxy{1125.0}{550.0}}]{$\dots$}
  \pgftext[right,at={\pgfpointxy{88.0}{550.0}}]{$3:$}
  \pgftext[at={\pgfpointxy{1050.0}{250.0}}]{9}
  \pgftext[at={\pgfpointxy{1050.0}{350.0}}]{9}
  \pgftext[at={\pgfpointxy{1050.0}{450.0}}]{9}
  \pgftext[at={\pgfpointxy{1050.0}{550.0}}]{9}
  \pgftext[top,right,at={\pgfpointxy{340.507}{223.575}}]{\scrs$f(0)$}
  \pgftext[top,at={\pgfpointxy{450.0}{214.188}}]{\scrs$f(1)$}
  \pgftext[top,left,at={\pgfpointxy{962.882}{227.16}}]{\scrs$f(2)$}
  \pgftext[bottom,right,at={\pgfpointxy{238.359}{371.767}}]{\scrs$f(3)$}
  \pgftext[bottom,at={\pgfpointxy{450.0}{385.812}}]{\scrs$f(4)$}
  \pgftext[bottom,at={\pgfpointxy{650.0}{385.812}}]{\scrs$f(5)$}
  \pgftext[top,left,at={\pgfpointxy{962.882}{327.16}}]{\scrs$f(6)$}
  \pgftext[bottom,right,at={\pgfpointxy{438.359}{571.767}}]{\scrs$f(7)$}
  \pgftext[bottom,at={\pgfpointxy{1050.0}{585.812}}]{\scrs$f(8)$}
\end{pgfpicture}
  \caption{An embedding $9 \hookrightarrow \omega \cdot 4$}
  \label{fbrd-scat.fig.ex1}
\end{figure}

\begin{EX}\label{fbrd-scat.ex.alpha-cdot-m-1}
  Consider the embedding $f : 9 \hookrightarrow \omega \cdot 4$ given by
  $$
    f = \left(
      \begin{matrix}
        0     & 1     & 2     & 3     & 4     & 5     & 6     & 7     & 8     \\
        (2,0) & (3,0) & (8,0) & (1,1) & (3,1) & (5,1) & (8,1) & (3,3) & (9,3)
      \end{matrix}
    \right)
  $$
  (see Fig.~\ref{fbrd-scat.fig.ex1}). Then $\val(f) = \{1 < 2 < 3 < 5 < 8 < 9\}$,
  while
  $\tp(f) = (3, 4, 0, 2, \sigma)$ where $\sigma = \{3 \prec 0 \prec 1 \equiv 4 \equiv 7 \prec 5 \prec 2 \equiv 6 \prec 8\}$.
  Clearly, $\equiv_\sigma$ has 6 blocks $\{3 | 0 | 1, 4, 7 | 5 | 2, 6 | 8\}$, and the length of the chain $\val(f)$ is~6.
\end{EX}

\begin{EX}\label{fbrd-scat.ex.alpha-cdot-m-2}
  Let $n = 7$ and $m = 5$, and consider the pair $(\tau, V)$ where $V = \{13 < 19 < 25 < 43\}$
  and $\tau = (0, 2, 1, 0, 4, \{3 \prec 1 \prec 0 \equiv 2 \equiv 6 \prec 4 \equiv 5\})$.
  Let us show that there is a unique embedding $f : 7 \hookrightarrow \omega \cdot 5$ such that
  $(\tp(f), \val(f)) = (\tau, V)$. From the sequence $(0, 2, 1, 0, 4)$ we can easily reconstruct
  the levels:
  $$
    f = \left(
      \begin{matrix}
        0       & 1       & 2       & 3       & 4       & 5       & 6 \\
        (?, 1)  & (?, 1)  & (?, 2)  & (?, 4)  & (?, 4)  & (?, 4)  & (?, 4) 
      \end{matrix}
    \right),
  $$
  while $V$ and $\sigma$ provide information on the values: $f(3)$ has the smallest value, so it has to be 13,
  $f(1)$ takes the next value, 19, then come $f(0)$, $f(2)$ and $f(6)$ with the same value 25, and finally
  $f(4)$ and $f(5)$ have the same value 43:
  $$
    f = \left(
      \begin{matrix}
        0       & 1       & 2       & 3       & 4       & 5       & 6 \\
        (25, 1) & (19, 1) & (25, 2) & (13, 4) & (43, 4) & (43, 4) & (25, 4) 
      \end{matrix}
    \right).
  $$
\end{EX}

\begin{LEM}\label{fbrd-scat.lem.1}
  Let $A$ be a countable well-ordered set with finite big Ramsey degrees.
  Fix integers $n \ge 1$ and $m \ge 1$. Then for every $(n, m)$-multiplicative type $\tau$,
  every $k \ge 2$ and every coloring
  $
    \chi : \Emb_\tau(n, A \cdot m) \to k
  $
  there is a $U \subseteq A$ order-isomorphic to $A$ such that
  $$
    |\chi(\Emb_\tau(n, U \cdot m))| \le T(r(\tau), A).
  $$
\end{LEM}
\begin{proof}
  Fix $n$, $m$ and $\tau$ as in the formulation of the lemma and let $r = r(\tau)$.
  Recall that each embedding $f : n \hookrightarrow A \cdot m$ is uniquely determined
  by the pair $(\tp(f), \val(f))$. Moreover, given a $v \in \Emb(r, A)$ there is a unique
  embedding $f_v : n \hookrightarrow A \cdot m$ with $\tp(f_v) = \tau$ and $\val(f_v) = v$. Therefore,
  $$
    \val : \Emb_\tau(n, A \cdot m) \to \Emb(r, A)
  $$
  is a bijection.
  Take any $k \ge 2$, a coloring $\chi : \Emb_\tau(n, A \cdot m) \to k$,
  and construct a coloring $\chi' : \Emb(r, A) \to k$ by
  $\chi'(v) = \chi(f_v)$. As we have just seen, this coloring is well defined.
  Since $T(r, A) < \infty$ there is a $U \subseteq A$ order-isomorphic to $A$
  and a set of colors $C \subseteq k$ such that $|C| \le T(r, A)$ and
  $\chi'(\Emb(r, U)) \subseteq C$. But then it is easy to see that
  $$
    \chi(\Emb_\tau(n, U \cdot m)) \subseteq C
  $$
  because $\chi(f) = \chi'(\val(f))$ for all $f \in \Emb_\tau(n, U \cdot m)$.
\end{proof}

\begin{THM}\label{fbrd-scat.thm.alpha-cdot-m}
  Let $\alpha$ be a countable ordinal with finite big Ramsey degrees.
  Then $\alpha \cdot m$ has finite big Ramsey degrees for every integer $m \ge 1$.
\end{THM}
\begin{proof}
  Fix $n \ge 1$ and $m \ge 1$ and let $\tau_0$, $\tau_1$, \dots, $\tau_{t-1}$ be all the $(n, m)$-multiplicative types.
  We are going to show that
  $$
    T(n, \alpha \cdot m) \le \sum_{j < t} T(r(\tau_j), \alpha) < \infty.
  $$
  Take any $k \ge 2$ and any coloring $\chi : \Emb(n, \alpha \cdot m) \to k$.
  By Lemma~\ref{fbrd-scat.lem.1} there is a $U_0 \subseteq \alpha$ order-isomorphic to $\alpha$ such that
  $$
    |\chi(\Emb_{\tau_0}(n, U_0 \cdot m))| \le T(r(\tau_0), \alpha).
  $$
  By the same lemma for each $j \in \{1, \ldots, t-1\}$
  we can inductively construct a $U_j \subseteq U_{j-1}$ order-isomorphic to $U_{j-1}$ (and hence to $\alpha$) such that
  $$
    |\chi(\Emb_{\tau_j}(n, U_j \cdot m))| \le T(r(\tau_j), \alpha).
  $$
  Then, having in mind that $U_{t-1} \subseteq U_j$,
  \begin{align*}
    |\chi(\Emb(n, U_{t-1} \cdot m))|
    &= \sum_{j < t} |\chi(\Emb_{\tau_j}(n, U_{t-1} \cdot m))|\\
    &\le \sum_{j < t} |\chi(\Emb_{\tau_j}(n, U_{j} \cdot m))| \le \sum_{j < t} T(r(\tau_j), \alpha). \qedhere
  \end{align*}
\end{proof}

Another consequence of Lemma~\ref{fbrd-scat.lem.1} is the following product Ramsey theorem.

\begin{THM}\label{fbrd-scat.thm.prod-ramsey}
  Let $\alpha$ be a countable ordinal with finite big Ramsey degrees.
  For any choice of integers $s \ge 1$ and $n_0, \ldots, n_{s-1} \ge 1$ there is an integer
  $t = t(\alpha, n_0, \ldots, n_{s-1})$ such that for every $k \ge 2$ and for every coloring
  $$
    \chi : \Emb(n_0, \alpha) \times \ldots \times \Emb(n_{s-1}, \alpha) \to k
  $$
  there is a $U \subseteq \alpha$ order-isomorphic to $\alpha$ such that
  $$
    |\chi(\Emb(n_0, U) \times \ldots \times \Emb(n_{s-1}, U))| \le t.
  $$
\end{THM}
\begin{proof}
  Let $N = n_0 + \ldots + n_{s-1}$ and let $Q$ be the set of all the $(N, s)$-multiplicative types of the form
  $(n_0, \ldots, n_{s-1}, \sigma)$ where $\sigma$ is arbitrary. Note that there are finitely many possibilities to choose $\sigma$
  so $Q$ is finite. Put
  $
    t = \sum_{\tau \in Q} T(r(\tau), \alpha)
  $.
  
  Take any $k \ge 2$ and any coloring
  $
    \chi : \Emb(n_0, \alpha) \times \ldots \times \Emb(n_{s-1}, \alpha) \to k
  $.
  It is easy to see that
  $$
    \Phi : \Emb(n_0, \alpha) \times \ldots \times \Emb(n_{s-1}, \alpha) \to \bigcup_{\tau \in Q} \Emb_\tau(N, \alpha \cdot s)
  $$
  given by $\Phi(f_0, \ldots, f_{s-1}) = f_0 + \ldots + f_{s-1}$ is a bijection. So, let us define
  $$
    \chi' : \bigcup_{\tau \in Q} \Emb_\tau(N, \alpha \cdot s) \to k
  $$
  by $\chi'(f) = \chi(\Phi^{-1}(f))$. We can now repeat the argument used in the proof of
  Theorem~\ref{fbrd-scat.thm.alpha-cdot-m} to show that there is a $U \subseteq \alpha$
  order-isomorphic to $\alpha$ such that
  $$
    \Big|\chi'\Big(\bigcup_{\tau \in Q} \Emb_\tau(N, U \cdot s)\Big)\Big| \le
    \sum_{\tau \in Q} T(r(\tau), \alpha) = t.
  $$
  Finally, note that the following restriction of $\Phi$:
  $$
    \Phi' : \Emb(n_0, U) \times \ldots \times \Emb(n_{s-1}, U) \to \bigcup_{\tau \in Q} \Emb_\tau(N, U \cdot s)
  $$
  (which takes $f$ to $\Phi(f)$, of course) is well-defined and a bijection. Therefore,
  $$
    \chi(\Emb(n_0, U) \times \ldots \times \Emb(n_{s-1}, U)) = \chi'\Big(\bigcup_{\tau \in Q} \Emb_\tau(N, U \cdot s)\Big),
  $$
  whence $|\chi(\Emb(n_0, U) \times \ldots \times \Emb(n_{s-1}, U))| \le t$.
\end{proof}

We conclude the section by specializing the above results for $\omega$.
Note, first, that due to the infinite version of the Ramsey's Theorem Lemma~\ref{fbrd-scat.lem.1} takes the following form

\begin{COR}\label{fbrd-scat.cor.1-omega}
  Fix integers $n \ge 1$ and $m \ge 1$.
  For every $(n, m)$-multi\-pli\-ca\-tive type $\tau$, every $k \ge 2$ and every coloring
  $
    \chi : \Emb_\tau(n, \omega \cdot m) \to k
  $
  there is an infinite $U \subseteq \omega$ such that
  $
    |\chi(\Emb_\tau(n, U \cdot m))| = 1
  $. \qednoproof
\end{COR}

An $(n, m)$-multiplicative type $\tau$ is \emph{strict} if $\sigma$ is a chain on $n$,
or, equivalently, if $\equiv_\sigma$ is the trivial relation $\{(i,i) : i \in n\}$.
Let $\Stp(n, m)$ denote the set of all the strict $(n, m)$-multiplicative types.

\begin{LEM}\label{fbrd-scat.lem.n1}
  $|\Stp(n,m)| = m^n$.
\end{LEM}
\begin{proof}
  Every strict $(n,m)$-multiplicative type $\tau = (p_0, p_1, \dots, p_{m-1}, \{i_0 \prec i_1 \prec \dots \prec i_{n-1}\})$
  can be represented by a word of length $n$ over the alphabet $m = \{0, 1, \dots, m-1\}$ as follows: in the sequence
  $(i_0, i_1, \dots, i_{n-1})$ replace $0, \dots, p_0-1$ with 0 if $p_0 > 0$,
  then replace $p_0, \dots, p_0+p_1-1$ with 1 if $p_1 > 0$,
  then replace $p_0+p_1, \dots, p_0+p_1+p_2-1$ with 2 if $p_2 > 0$, and so on.
  
  For example, for $n = 7$, $m = 4$ and $\tau = (2, 0, 4, 1, \{2 \prec 6 \prec 0 \prec 3 \prec 4 \prec 1 \prec 5\})$
  the sequence $(2, 6, 0, 3, 4, 1, 5)$ is transformed by replacing 0 and 1 with 0, then 2, 3, 4 and 5 with 2, and 6 by 3
  to get the word $2302202$.
  
  Conversely, any word $w$ of length $n$ over the alphabet $m = \{0, 1, \dots, m-1\}$ uniquely determines a strict
  $(n, m)$-multiplicative type.
  First, let $p_i$ be the number of occurrences of letter $i$ in $w$, $i < m$. Then transform  $w$ into a linear ordering of
  $n$ as follows: if $p_0 > 0$ replace all the 0's by integers $0$, $1$, \dots, $p_0-1$ going from left to right,
  then if $p_1 > 0$ replace all the 1's by integers $p_0$, $p_0+1$, \dots, $p_0+p_1-1$ going from left to right, and so on.
  For example, starting from $w = 2302202$ we get first $p_0 = 2$, $p_1 = 0$, $p_2 = 4$ and $p_3 = 1$, and then to get $\sigma$
  we replace the two 0's with 0 and 1 (from left to right), then the four 2's by 2, 3, 4 and 5, and finally the remaining 3 by 6 to get
  $2 \prec 6 \prec 0 \prec 3 \prec 4 \prec 1 \prec 5$.
  This correspondence is clearly bijective, which proves the claim.
\end{proof}

\begin{THM}\label{fbrd-scat.thm.omega.m}
  $T(n, \omega \cdot m) = m^n$, for all $n \ge 1$ and $m \ge 1$.
\end{THM}
\begin{proof}
  Having Lemma~\ref{fbrd-scat.lem.n1} in mind it suffices to show that
  $T(n, \omega \cdot m) = |\Stp(n, m)|$.
  Let $\tau_0$, $\tau_1$, \dots, $\tau_{t-1}$ be all the $(n,m)$-multiplicative types (not necessarily
  strict).

  Let us first show that $T(n, \omega \cdot m) \le |\Stp(n, m)|$.
  Take any $k \ge 2$ and any coloring $\chi : \Emb(n, \omega \cdot m) \to k$.
  By Corollary~\ref{fbrd-scat.cor.1-omega} there is an infinite $S_0 \subseteq \omega$ such that
  $|\chi(\Emb_{\tau_0}(n, S_0 \cdot m))| = 1$.
  By the same for each $j \in \{1, \ldots, t-1\}$ we can inductively construct an infinite $S_j \subseteq S_{j-1}$ such that
  $|\chi(\Emb_{\tau_j}(n, S_j \cdot m))| = 1$. So,
  $$
    |\chi(\Emb(n, S_{t-1} \cdot m))| \le t.
  $$
  Let $S_{t-1} = \{s_0 < s_1 < s_2 < \dots \}$. Define $U_0$, $U_1$, \dots, $U_{m-1}$ as follows:
  \begin{align*}
    U_0 &= \{s_0 < s_m < s_{2m} < \dots \},\\
    U_1 &= \{s_1 < s_{m+1} < s_{2m+1} < \dots \},\\
        &\vdots\\
    U_{m-1} &= \{s_{m-1} < s_{2m-1} < s_{3m-1} < \dots \}.
  \end{align*}
  Then for each $i < t$ we have that
  $$
    |\chi(\Emb_{\tau_{i}}(n, U_0 + U_1 + \dots + U_{m-1}))| \le 1,
  $$
  because $U_0 + U_1 + \dots + U_{m-1} \subseteq S_{t-1} \cdot m$. Moreover,
  if the type $\tau_i$ is not strict then
  $$
    |\chi(\Emb_{\tau_{i}}(n, U_0 + U_1 + \dots + U_{t-1}))| = 0
  $$
  because there do not exist two identical values on different levels in
  $U_0 + U_1 + \dots + U_{m-1}$. Hence,
  $$
    |\chi(\Emb(n, U_0 + U_1 + \dots + U_{m-1}))| \le |\Stp(n,m)|.
  $$

  To conclude the proof let us show that $T(n, \omega \cdot m) \ge |\Stp(n, m)|$.
  Let $\Stp(n, m) = \{\tau^*_0, \dots, \tau^*_{s-1}\}$ where $s = |\Stp(n, m)|$.
  Define the coloring $\chi^* : \Emb(n, \omega \cdot n) \to s$ by
  $$
    \chi^*(f) = \begin{cases}
      i, & \tp(f) = \tau^*_i,\\
      0, & \text{otherwise (i.e. $\tp(f)$ is not strict).}
    \end{cases}
  $$
  It is obvious that for arbitrary infinite $U_0 \subseteq \omega$,
  $U_1 \subseteq \omega$, \dots, $U_{m-1} \subseteq \omega$ we have that
  $$
    |\chi^*(\Emb(n, U_0 + U_1 + \dots + U_{m-1}))| = s
  $$
  because all the strict types are realized in $\Emb(n, U_0 + U_1 + \dots + U_{m-1})$.
\end{proof}

By specializing Theorem~\ref{fbrd-scat.thm.prod-ramsey} to $\omega$ we get the following
infinite version of the Product Ramsey Theorem,
where $\Stirling nk$ denotes the Stirling number of the second kind.

\begin{COR}\label{fbrd-scat.cor.inf-prod-ramsey}
  Fix integers $s \ge 1$ and $n_0, \ldots, n_{s-1} \ge 1$. Then for every $k \ge 2$ and for every coloring
  $
    \chi : \binom\omega{n_0} \times \ldots \times \binom\omega{n_{s-1}} \to k
  $
  there is an infinite $U \subseteq \omega$ such that
  $
    \left|\chi\left(\binom U{n_0} \times \ldots \times \binom U{n_{s-1}}\right)\right| \le
    \sum_{j=1}^N j! \Stirling Nj,
  $
  where $N = n_0 + n_1 + \ldots + n_{s-1}$.

  This upper bound is tight in the following sense: for $k = \sum_{j=1}^N j! \Stirling Nj$ there is a coloring
  $
    \chi^* : \binom\omega{n_0} \times \ldots \times \binom\omega{n_{s-1}} \to k
  $
  such that for every infinite $U \subseteq \omega$ we have that
  $
    \left|\chi^*\left(\binom U{n_0} \times \ldots \times \binom U{n_{s-1}}\right)\right| =
    \sum_{j=1}^N j! \Stirling Nj.
  $
\end{COR}
\begin{proof}
  Let $Q$ be the set of all the $(N, s)$-multiplicative types of the form
  $(n_0, \ldots, n_{s-1}, \sigma)$ where $\sigma$ is arbitrary.
  As in the proof of Theorem~\ref{fbrd-scat.thm.prod-ramsey} we conclude that there is an infinite $U \subseteq \omega$
  such that
  $$
    |\chi(\Emb(n_0, U) \times \ldots \times \Emb(n_{s-1}, U))| \le \sum_{\tau \in Q} T(r(\tau), \omega) = |Q|
  $$
  by the infinite version of the Ramsey's Theorem.
  But then it is easy to see that $|Q| = \sum_{j=1}^N j! \Stirling Nj$, the $N$th ordered Bell number minus
  the 0th term.
  
  To show that the upper bound is tight let us enumerate $Q$ as $\{\tau_0$, $\tau_1$, \ldots, $\tau_{k-1}\}$.
  Let
  $$
    \Phi : \Emb(n_0, \omega) \times \ldots \times \Emb(n_{s-1}, \omega) \to \bigcup_{\tau \in Q} \Emb_\tau(N, \omega \cdot s)
  $$
  be the bijection as in the proof of Theorem~\ref{fbrd-scat.thm.prod-ramsey} and define
  $$
    \chi^* : \Emb(n_0, \omega) \times \ldots \times \Emb(n_{s-1}, \omega) \to k
  $$
  by $\chi^*(f_0, f_1, \ldots, f_{s-1}) = j$ if and only if $\tp(\Phi(f_0, f_1, \ldots, f_{s-1})) = \tau_j$.
  Now, take an arbitrary infinite $U \subseteq \omega$ and let us show that
  $$
    \Phi(\Emb(n_0, U) \times \ldots \times \Emb(n_{s-1}, U)) = \bigcup_{\tau \in Q} \Emb_\tau(N, U \cdot s)
  $$
  realizes all the types from~$Q$. In other words, let us show that $\Emb_\tau(N, U \cdot s) \ne \0$ for every $\tau \in Q$.
  But that is straightforward.
  Take any $\tau = (n_0, \ldots, n_{s-1}, \sigma) \in Q$, let $r = r(\tau)$
  and let $u_0 < u_1 < \ldots < u_{r-1}$ be an arbitrary $r$-element chain of elements of $U$.
  As we have demonstrated in Section~\ref{fbrd-scat.sec.alpha-cdot-m}
  there is a unique $g \in \Emb(N, \omega \cdot s)$ with $\tp(g) = \tau$ and $\val(g) = \{u_0, u_1, \ldots, u_{r-1}\}$.
  Then, clearly, $g \in \Emb_\tau(N, U \cdot s)$.
\end{proof}

\section{Raising to a finite power}
\label{fbrd-scat.sec.alpha-pow-m}

Let $A$ be a well-ordered set.
Since $A^m$ is order-isomorphic to $(\underbrace{A \times \ldots \times A}_m, \Boxed\alex)$,
for every embedding $f : n \hookrightarrow A^m$ and every $i < n$ we can consider $f(i)$ to be
an $m$-tuple of elements of $A$:
\begin{align*}
    f(0) &= (a_{00}, a_{01}, \ldots, a_{0,m-1}),\\
    f(1) &= (a_{10}, a_{11}, \ldots, a_{1,m-1}),\\
         &\;\;\vdots \\
  f(n-1) &= (a_{n-1,0}, a_{n-1,1}, \ldots, a_{n-1,m-1}).
\end{align*}
where the $m$-tuples are ordered antilexicographically:
$$
  (a_{00}, a_{01}, \ldots, a_{0,m-1}) \alex \ldots \alex (a_{n-1,0}, a_{n-1,1}, \ldots, a_{n-1,m-1}).
$$
This, in turn, means that we can represent each embedding $f : n \hookrightarrow A^m$ as an ordered tree
$\tree(f)$ of height $m$ with exactly $n$ leaves where all the vertices of the tree except for the root are
labelled by the elements of~$A$.
The fact that $\tree(f)$ is ordered means that for every vertex of the tree which is not a leaf, the labels
of all the immediate successors of the vertex form finite chains of elements of~$A$.
Let $\tp(f)$ denote the unlabelled version of $\tree(f)$. An unlabelled tree of height $m$ and with
exactly $n$ leaves will be referred to as an \emph{$(n, m)$-power type}.
For an $(n, m)$-power type $\tau$ let
$$
  \Emb_\tau(n, A^m) = \{ f \in \Emb(n, A^m) : \tp(f) = \tau \}.
$$

\begin{EX}\label{fbrd-scat.ex.pow-1}
  Let $f : 12 \hookrightarrow \omega^4$ be the following embedding:
  $$
    \begin{array}{r@{\,=\,}c}
    f(0) & (0,1,0,0)\\
    f(1) & (3,1,0,0)\\
    f(2) & (1,3,6,0)\\
    f(3) & (1,3,6,0)
    \end{array}
    \quad
    \begin{array}{r@{\,=\,}c}
    f(4) & (5,7,0,2)\\
    f(5) & (0,8,1,2)\\
    f(6) & (1,1,3,2)\\
    f(7) & (4,0,2,5)
    \end{array}
    \quad
    \begin{array}{r@{\,=\,}c}
    f(8)  & (1,1,2,5)\\
    f(9)  & (3,1,2,5)\\
    f(10) & (5,1,4,5)\\
    f(11) & (7,2,4,5)
    \end{array}
  $$
  Fig.~\ref{fbrd-scat.fig.2}~$(a)$ depicts $\tree(f)$, and
  Fig.~\ref{fbrd-scat.fig.2}~$(b)$ depicts $\tp(f)$.
\end{EX}

\begin{figure}
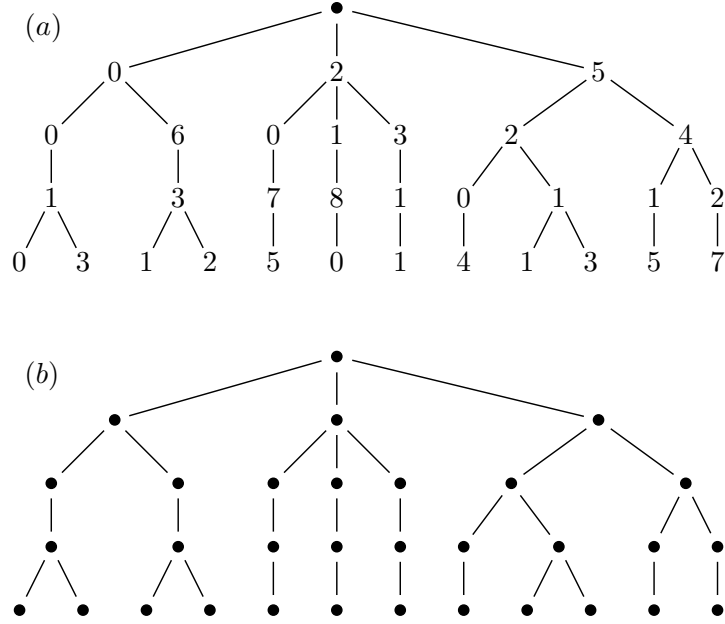

  \centering
  \input fbrd-scat-ex2.pgf
  \caption{A tree representation and a type of an embedding $12 \hookrightarrow \omega^4$}
  \label{fbrd-scat.fig.2}
\end{figure}

For an embedding $f : n \hookrightarrow A^m$ let $\val(f)$ denote the tuple of all the finite subchains
of $A$ that appear as ordered immediate successors of non-leaves of $\tree(f)$, where the subchains are
listed from top to bottom, and from left to right.
Given finite $n$ and $m$, each embedding $f : n \hookrightarrow A^m$ is uniquely determined
by the pair $(\tp(f), \val(f))$.

\begin{EX}
  For the embedding $f : 12 \hookrightarrow \omega^4$ in Example~\ref{fbrd-scat.ex.pow-1} we have that
  \begin{align*}
    \val(f) = (&0 < 2 < 5,\quad 0 < 6,\quad 0 < 1 < 3,\quad 2 < 4,\quad 1,\quad 3,\\
               &7,\quad 8,\quad 1,\quad 0 < 1,\quad 1 < 2,\quad 0 < 3,\quad 1 < 2,\quad 5,\\
               &0,\quad 1,\quad 4,\quad 1 < 3,\quad 5,\quad 7).
  \end{align*}
  Clearly, given an $(n,m)$-power type such as the one in Fig.~\ref{fbrd-scat.fig.2}~$(b)$ and an appropriate
  sequence of finite subchains of $\alpha$ such as the one above, one can uniquely reconstruct the tree of the embedding,
  and hence the embedding itself.
\end{EX}

\begin{LEM}\label{fbrd-scat.lem.1-power}
  Let $A$ be a countable well-ordered set with finite big Ramsey degrees.
  Fix integers $n \ge 1$ and $m \ge 1$. Then for every $(n, m)$-power type $\tau$ there is an integer $t = t(\tau)$
  such that every $k \ge 2$ and every coloring
  $$
    \chi : \Emb_\tau(n, A^m) \to k
  $$
  there is a $U \subseteq A$ order-isomorphic to $A$ such that
  $$
    |\chi(\Emb_\tau(n, U^m))| \le t.
  $$
\end{LEM}
\begin{proof}
  Fix $n$, $m$ and $\tau$ as in the formulation of the lemma.
  Recall that each embedding $f : n \hookrightarrow A^m$ is uniquely determined
  by the pair $(\tp(f), \val(f))$. Moreover, for appropriately chosen $n_0, n_1, \ldots, n_{s-1} \ge 1$ (that depend on $\tau$ only),
  given a $v \in \Emb(n_0, A) \times \Emb(n_1, A) \times \ldots \times \Emb(n_{s-1}, A)$
  there is a unique embedding $f_v : n \hookrightarrow A^m$ with $\tp(f_v) = \tau$ and $\val(f_v) = v$.
  Therefore,
  $$
    \val : \Emb_\tau(n, A^m) \to \Emb(n_0, A) \times \Emb(n_1, A) \times \ldots \times \Emb(n_{s-1}, A)
  $$
  is a bijection. Let $t$ be an integer whose existence is guaranteed by Theorem~\ref{fbrd-scat.thm.prod-ramsey}.

  Take any $k \ge 2$, a coloring $\chi : \Emb_\tau(n, A \cdot m) \to k$,
  and construct a coloring $\chi' : \Emb(n_0, A) \times \Emb(n_1, A) \times \ldots \times \Emb(n_{s-1}, A) \to k$ by
  $\chi'(v) = \chi(f_v)$. As we have just seen, this coloring is well defined.
  By Theorem~\ref{fbrd-scat.thm.prod-ramsey} there is a $U \subseteq A$ order-isomorphic to $A$ such that
  $$
    |\chi'(\Emb(n_0, U) \times \ldots \times \Emb(n_{s-1}, U))| \le t.
  $$
  But then it is easy to see that
  $
    |\chi(\Emb_\tau(n, U^m))| \le t
  $
  because $\chi(f) = \chi'(\val(f))$ for all $f \in \Emb_\tau(n, U^m)$.
\end{proof}

\begin{THM}\label{fbrd-scat.thm.alpha-pow-m}
  Let $\alpha$ be a countable ordinal with finite big Ramsey degrees.
  Then $\alpha^m$ has finite big Ramsey degrees for every integer $m \ge 1$.
\end{THM}
\begin{proof}
  Fix $n \ge 1$ and $m \ge 1$ and let $\tau_0$, $\tau_1$, \dots, $\tau_{s-1}$ be all the $(n, m)$-power types.
  Let $t_j = t(\tau_j)$, $j < s$, be the integers whose existence is provided by Lemma~\ref{fbrd-scat.lem.1-power}.
  We are going to show that
  $$
    T(n, \alpha^m) \le \sum_{j < s} t_j < \infty.
  $$
  Take any $k \ge 2$ and any coloring $\chi : \Emb(n, \alpha^m) \to k$.
  By Lemma~\ref{fbrd-scat.lem.1-power} there is a $U_0 \subseteq \alpha$ order-isomorphic to $\alpha$ such that
  $$
    |\chi(\Emb_{\tau_0}(n, U_0^m))| \le t_0.
  $$
  By the same for each $j \in \{1, \ldots, s-1\}$ a $U_j \subseteq U_{j-1}$ order-isomorphic to $U_{j-1}$ (and hence to $\alpha$) such that
  $$
    |\chi(\Emb_{\tau_j}(n, U_j^m))| \le t_j.
  $$
  Then using the fact that $U_{s-1} \subseteq U_j$ we have:
  \begin{align*}
    |\chi(\Emb(n, U_{s-1}^m))|
    &= \sum_{j < s} |\chi(\Emb_{\tau_j}(n, U_{s-1}^m))|\\
    &\le \sum_{j < s} |\chi(\Emb_{\tau_j}(n, U_{j}^m))| \le \sum_{j < t} t_j. \qedhere
  \end{align*}
\end{proof}

\section{The main result}
\label{fbrd-scat.sec.ord-pol}

In this section we prove the main result of the paper: we show that a countable
ordinal $\alpha$ has finite big Ramsey degrees if and only if $\alpha < \omega^\omega$.
For countable ordinals $\alpha \ge \omega^\omega$ we show that all their big Ramsey degrees are infinite.

Let $\alpha$ be an ordinal and let $\xi < \alpha$. Then $\alpha \setminus \xi$ will be referred to
as a \emph{remainder of $\alpha$}. Note that for every ordinal $\alpha \ge 1$ every remainder of $\omega^\alpha$ is
order-isomorphic to~$\omega^\alpha$.

\begin{LEM}\label{fbrd-scat.lem.emb-sep}
  Let $\alpha_0 \ge \alpha_1 \ge \ldots \ge \alpha_{n-1}$ be ordinals such that for each $i < n - 1$ every remainder of $\alpha_i$
  is order-isomorphic to~$\alpha_i$ (note that we do not require this for $\alpha_{n-1}$). Then
  $
    \Emb\Big(\sum_{i < n} \alpha_i, \sum_{i < n} \alpha_i\Big) = \sum_{i < n} \Emb(\alpha_i, \alpha_i)
  $.
\end{LEM}
\begin{proof}
  The inclusion $\supseteq$ is obvious. As for the other inclusion
  take any $f \in \Emb(\sum_{i < n} \alpha_i, \sum_{i < n} \alpha_i)$.
  
  \medskip
  
  Claim 1. For all $i < n$ we have that $f(0, i) = (\xi, i)$ for some $\xi \in \alpha_i$.
  
  \medskip
  
  Proof. Suppose, first, that there is an $i$ and a $j > i$ such that $f(0, i) = (\xi, j)$ for some $\xi \in \alpha_j$.
  Then there exist $p$ and $q$ such that $p < q$, $f(0, p) = (\epsilon, q)$ and $f(0, q) = (\eta, q)$ for some $\epsilon, \eta \in \alpha_q$
  satisfying $\epsilon < \eta$.
  But then $f$ restricted to $\alpha_p \times \{p\}$ is actually an embedding of $\alpha_p$ into $\eta < \alpha_q$ whence follows that
  $\alpha_p < \alpha_q$. Contradiction with the assumption that $\alpha_p \ge \alpha_q$.
  
  Suppose, now, that there is an $i$ and a $j < i$ such that $f(0, i) = (\xi, j)$ for some $\xi \in \alpha_j$.
  Then there exist $p$ and $q$ such that $p < q$, $f(0, p) = (\epsilon, p)$ and $f(0, q) = (\eta, p)$ for some $\epsilon, \eta \in \alpha_p$
  satisfying $\epsilon < \eta$.
  But then $f$ restricted to $\alpha_p \times \{p\}$ is actually an embedding of $\alpha_p$ into $\eta < \alpha_p$ whence follows that
  $\alpha_p < \alpha_p$. Contradiction.
  This concludes the proof of Claim~1.

  \medskip
  
  Claim 2. For all $i < n$ and all $\xi \in \alpha_i$ we have that $f(\xi, i) \in \alpha_i \times \{i\}$.
  
  \medskip
  
  Proof. Suppose this is not the case. Then there is an $i < n$ and a $\xi \in \alpha_i$ such that $f(\xi, i) \notin \alpha_i \times \{i\}$.
  Then $f(\xi, i) \in \alpha_{i+1} \times \{i+1\}$ since $(0, i) < (\xi, i) < (0, i+1)$,
  $f(0, i) \in \alpha_i \times \{i\}$ and $f(0, i+1) \in \alpha_{i+1} \times \{i+1\}$. Let $f(0, i+1) = (\eta, i+1)$ and
  let $\rho = \alpha_i \setminus \xi$ be a remainder of $\alpha_i$.
  Then $f$ restricted to $\rho \times \{i\}$ is actually an embedding of $\rho$ into $\eta$.
  Since $\rho$ is order-isomorphic to $\alpha_i$ by the assumption, we finally get that $\alpha_i \le \eta < \alpha_{i+1}$. Contradiction.
  This concludes the proof of Claim~2.
  
  \medskip
  
  It is now easy to complete the proof. Claim~2 yields that $\restr{f}{\alpha_i \times \{i\}}$ is an embedding
  $\alpha_i \times \{i\} \hookrightarrow \alpha_i \times \{i\}$, so it uniquely determines an embedding
  $f_i \in \Emb(\alpha_i, \alpha_i)$
  by $f_i(\xi) = \eta$ if and only if $f(\xi, i) = (\eta, i)$. Then it clearly follows that
  $f = \sum_{i < n} f_i \in \sum_{i < n} \Emb(\alpha_i, \alpha_i)$.
\end{proof}

\begin{LEM}\label{fbrd-scat.lem.sub-sum}
  Let $\alpha_0 \ge \alpha_1 \ge \ldots \ge \alpha_{\ell-1}$ be ordinals such that for each $i < \ell - 1$ every remainder of $\alpha_i$
  is order-isomorphic to~$\alpha_i$ (note that we do not require this for $\alpha_{\ell-1}$).
  If $\alpha_0 + \alpha_1 + \ldots + \alpha_{\ell-1}$ has finite
  big Ramsey degrees then $\alpha_{j_0} + \alpha_{j_1} + \ldots + \alpha_{j_{m-1}}$ has finite
  big Ramsey degrees for all $m \ge 1$ and $0 \le j_0 < j_1 < \ldots < j_{m-1} \le \ell-1$.
\end{LEM}
\begin{proof}
  Assume that $\alpha_{0} + \alpha_{1} + \ldots + \alpha_{\ell-1}$ has finite big Ramsey degrees and take any
  $m \ge 1$ and $0 \le j_0 < j_1 < \ldots < j_{m-1} \le \ell-1$.
  Let $\alpha = \alpha_{0} + \alpha_{1} + \ldots + \alpha_{\ell-1}$,
  $\beta = \alpha_{j_0} + \alpha_{j_1} + \ldots + \alpha_{j_{m-1}}$ and
  let $\phi : \beta \hookrightarrow \alpha$ be the obvious embedding $\phi(\xi, s) = (\xi, j_s)$.
  We are going to show that
  $
    T(n, \beta) \le T(n, \alpha)
  $
  for all $n \ge 1$.
  
  Take any $n \ge 1$, $k \ge 2$ and a coloring
  $
    \chi : \Emb(n, \beta) \to k.
  $
  Define
  $
    \chi' : \Emb(n, \alpha) \to k
  $
  as follows: $\chi'(\phi \circ f) = \chi(f)$ for all $f \in \Emb(n, \beta)$,
  and $\chi'(g) = 0$ if there is no $f \in \Emb(n, \beta)$ with
  $g = \phi \circ f$. By the assumption, there is a $w \in \Emb(\alpha, \alpha)$
  such that
  $$
    |\chi'(w \circ \Emb(n, \alpha))| \le T(n, \alpha).
  $$
  Lemma~\ref{fbrd-scat.lem.emb-sep} now implies that there exist embeddings $w_i \in \Emb(\alpha_i, \alpha_i)$,
  $i < \ell$, such that $w = w_0 + w_1 + \ldots + w_{\ell-1}$.
  Let $w' = w_{j_0} + w_{j_1} + \ldots + w_{j_{m-1}}$. Clearly, $w' \in \Emb(\beta, \beta)$
  and $\phi \circ w' = w \circ \phi$. So,
  \begin{align*}
    |\chi(w' \circ \Emb(n, \beta))|
      &= |\chi'(\phi \circ w' \circ \Emb(n, \beta))| && \text{[definition of $\chi'$]}\\
      &= |\chi'(w \circ \phi \circ \Emb(n, \beta))| && \text{[$\phi \circ w' = w \circ \phi$]}\\
      &\le |\chi'(w \circ \Emb(n, \alpha))| \le T(n, \alpha). && \qedhere
  \end{align*}
\end{proof}

\begin{LEM}\label{fbrd-scat.lem.infty}
   $T(n, \omega^\alpha) = \infty$ for every $2 \le n < \omega$ and every countable ordinal~$\alpha \ge \omega$.
\end{LEM}
\begin{proof}
  Since $\omega^\alpha$ is a limit ordinal Lemma~\ref{fbrd-scat.lem.T-monotono} implies that it suffices
  to show the statement in case $n = 2$. So, let us show that
  $T(2, \omega^\alpha) = \infty$ for every countable ordinal~$\alpha \ge \omega$.

  Let $\alpha$ be a countable ordinal such that~$\alpha \ge \omega$. Since $\omega^\alpha$ is a countable
  scattered chain Theorem~\ref{fbrd-scat.thm.bracket} applies, so
  $\omega^\alpha \not\longrightarrow [\omega^{n_0}, \omega^{n_1}, \omega^{n_2}, \omega^{n_3}, \omega^{n_4}, \ldots]^2$,
  where $n_0 = 1$, $n_1 = n_2 = 2$, $n_3 = n_4 = 3$, and so on.
  Therefore, there exists a coloring
  $\gamma : \Emb(2, \omega^\alpha) \to \omega$ with the following property: for every $i < \omega$ and every
  subchain $H \subseteq \omega^\alpha$ such that $H \cong \omega^{n_i}$ we have that $i \in \gamma(\Emb(2, H))$.
  
  Now, take any $t \ge 2$ and consider the coloring $\chi_t : \Emb(2, \omega^\alpha) \to t$ given by
  $\chi_t(f) = \min\{t-1, \gamma(f)\}$. Let $S$ be an arbitrary subchain of $\omega^\alpha$ order-isomorphic to $\omega^\alpha$.  
  Since $\alpha \ge \omega$, for every $i < t$ there is a subchain $H_i \subseteq S$ order-isomorphic to $\omega^{n_i}$.
  By the construction of $\chi_t$ it then follows that $i \in \chi_t(\Emb(2, H_i)) \subseteq \chi_t(\Emb(2, S))$.
  Therefore, $|\chi_t(\Emb(2, S))| \ge t$.
\end{proof}

\begin{THM}\label{fbrd-scat.thm.MAIN}
  Let $\alpha$ be a countable ordinal.
  
  $(a)$ If $\alpha < \omega^\omega$ then $T(n, \alpha) < \infty$ for all $2 \le n < \omega$.

  $(b)$ If $\alpha \ge \omega^\omega$ then $T(n, \alpha) = \infty$ for all $2 \le n < \omega$.
\end{THM}
\begin{proof}
  $(a)$
  By Theorem~\ref{fbrd-scat.thm.alpha+m} it suffices to show that every ordinal of the form
  $$
    \omega^{d_0} \cdot c_0 + \omega^{d_1} \cdot c_1 + \ldots + \omega^{d_{k-1}} \cdot c_{k-1}
  $$
  where $d_0 > d_1 > \ldots > d_{k-1} \ge 1$ and $c_0, c_1, \ldots, c_{k-1} \ge 1$
  has finite big Ramsey degrees.
  
  Let $m = \max\{c_0, c_1, \ldots, c_{k-1}\}$.
  Theorems~\ref{fbrd-scat.thm.alpha+m}, \ref{fbrd-scat.thm.alpha-cdot-m} and \ref{fbrd-scat.thm.alpha-pow-m} ensure that
  $(\omega \cdot m + 1)^{d_0}$ has finite big Ramsey degrees. Using the fact that $(\omega \cdot m + 1) \cdot \omega = \omega^2$
  it is easy to see that
  \begin{gather*}
      (\omega \cdot m + 1)^{d_0} = \omega^{d_0} \cdot m + \omega^{d_0 - 1} \cdot m + \ldots + \omega^2 \cdot m + \omega \cdot m + 1 =\\
      = \underbrace{\omega^{d_0} + \ldots + \omega^{d_0}}_m + \underbrace{\omega^{d_0 - 1} + \ldots + \omega^{d_0 - 1}}_m + \ldots +
        \underbrace{\omega + \ldots + \omega}_m + 1,
  \end{gather*}
  and the latter sum of ordinals has finite big Ramsey degrees, as we have just seen.
  Having in mind that every remainder of $\omega^d$ is order-isomorphic to $\omega^d$ whenever $d \ge 1$,
  Lemma~\ref{fbrd-scat.lem.sub-sum} now yields that
  $$
    \underbrace{\omega^{d_0} + \ldots + \omega^{d_0}}_{c_0} + \underbrace{\omega^{d_1} + \ldots + \omega^{d_1}}_{c_1} + \ldots +
    \underbrace{\omega^{d_{k-1}} + \ldots + \omega^{d_{k-1}}}_{c_{k-1}}
  $$
  has finite big Ramsey degrees as well.

  \bigskip

  $(b)$
  Assume that $\alpha \ge \omega^\omega$ is a countable ordinal and let
  $
    \alpha = \omega^{\beta_0} \cdot c_0 + \omega^{\beta_1} \cdot c_1 + \ldots + \omega^{\beta_{k-1}} \cdot c_{k-1}
  $
  be the Cantor normal form of $\alpha$ where $\beta_0 > \beta_1 > \ldots > \beta_{k-1} \ge 0$ and
  $c_0, c_1, \ldots, c_{k-1} < \omega$. Then, clearly, $\beta_0 \ge \omega$.
  By Lemma~\ref{fbrd-scat.lem.infty} we know that $T(n, \omega^{\beta_0}) = \infty$
  for all $2 \le n < \omega$. Lemma~\ref{fbrd-scat.lem.sub-sum} then yields that
  $T(n, \alpha) = \infty$ for all $2 \le n < \omega$.
\end{proof}

\section{Concluding remarks}
\label{fbrd-scat.sec.conclusion}

Our ultimate goal is to extend the result of Laver \cite{laver-decomposition} and
characterize scattered countable chains which have finite big Ramsey degrees.
By a slight extension of the results of this paper we can resolve the issue
for a very particular class of scattered chains which are not ordinals.
As a consequence we do justice to $\ZZ$ by proving that
it, along with $\omega$ and $\QQ$, has finite big Ramsey degrees.

For a well-ordered chain $A$ let $A^*$ denote the chain $A$ with the order reversed. Moreover, for
$d \in \{\Boxed+, \Boxed-\}$ let
$$
  A^{(d)} = \begin{cases}
    A, & d = \Boxed+\\
    A^*, & d = \Boxed-.
  \end{cases}
$$

\begin{LEM}\label{fbrd-scat.lem.scat-omega-simple}
  For any $n, m \ge 1$ and any choice $d_0, d_1, \ldots, d_{m-1} \in \{\Boxed+, \Boxed-\}$,
  $$\textstyle
    T(n, \sum_{i<m} \omega^{(d_i)}) = T(n, \omega \cdot m).
  $$
\end{LEM}
\begin{proof}
  For an embedding $g : A \hookrightarrow B$ of chains let $g^*$ denote the obvious embedding $A^* \hookrightarrow B^*$,
  and for $d \in \{\Boxed+, \Boxed-\}$ let
  $$
    g^{(d)} = \begin{cases}
      g, & d = \Boxed+\\
      g^*, & d = \Boxed-.
    \end{cases}
  $$
  In particular, if $g : n \hookrightarrow B$ we may safely take that $g^* : n \hookrightarrow B^*$ because $n^* \cong n$.
  
  Every $f : n \hookrightarrow \sum_{i<m} \omega^{(d_i)}$ can be in an obvious way represented as
  $f = \sum_{i<m} f_i$ where $f_i : n_i \hookrightarrow \omega^{(d_i)}$, $n_i \ge 0$ and $n = n_0 + \ldots + n_{m-1}$.
  It is easy to see, then, that $f' = \sum_{i<m} f_i^{(d_i)}$ is an embedding $n \hookrightarrow \omega \cdot m$.
  Keeping $n$ and $d_i$'s fixed, it is also easy to see that $f'' = f$.

  With all these technicalities set up we can now show that
  $$\textstyle
    T(n, \sum_{i<m} \omega^{(d_i)}) \le T(n, \omega \cdot m).
  $$
  Take any $k \ge 2$, any coloring $\chi : \Emb(n, \sum_{i<m} \omega^{(d_i)}) \to k$
  and define $\chi' : \Emb(n, \omega \cdot m) \to k$ by $\chi'(f) = \chi(f')$. Then there is a
  $U \subseteq \omega \cdot m$ order-isomorphic to $\omega \cdot m$ such that $|\chi'(\Emb(n, U))| \le T(n, \omega \cdot m)$.
  Since $U \cong \omega \cdot m$ there exist infinite $U_i \subseteq \omega$, $i < m$, such that $U = U_0 + \ldots + U_{m-1}$.
  Then $\sum_{i < m} U_i^{(d_i)} \cong \sum_{i<m} \omega^{(d_i)}$ and
  $$\textstyle
    \chi'(\Emb(n, U)) = \chi'(\Emb(n, \sum_{i<m} U_i)) = \chi(\Emb(n, \sum_{i<m} U_i^{(d_i)})).
  $$
  Therefore, $|\chi(\Emb(n, \sum_{i<m} U_i^{(d_i)}))| \le T(n, \omega \cdot m)$. The other inequality follows
  by analogous arguments.
\end{proof}

\begin{COR}
  $T(n, \ZZ) = 2^n$ for all $n \ge 1$.
\end{COR}
\begin{proof}
  Since $\ZZ \cong \omega^* + \omega$ we have that $T(n, \ZZ) = T(n, \omega^* + \omega) = T(n, \omega \cdot 2) = 2^n$
  using Lemma~\ref{fbrd-scat.lem.scat-omega-simple} and Theorem~\ref{fbrd-scat.thm.omega.m}.
\end{proof}

\section{Acknowledgements}

We would like to thank F.~Galvin and S.~Todor\v cevi\'c for their help and many insightful observations that led us to the solution of
the problem for all countable ordinals.

The first author gratefully acknowledges the support of the Grant No.\ 174019 of the Ministry of Education,
Science and Technological Development of the Republic of Serbia.

\end{document}